\documentclass{amsart}

\usepackage{amssymb}
\usepackage{amscd}

\title{Galois points for double-Frobenius nonclassical curves}
\author{Herivelto Borges and Satoru Fukasawa}

\address{Universidade de S\~ao Paulo, Inst. de Ci\^encias Matem\'aticas
 e de Computa\c c\~ao, S\~ao Carlos, SP 13560-970,  Brazil.} 
\email{hborges@icmc.usp.br}
\thanks{The first author was partially supported by FAPESP grant 2017/04681-3}

\address{Department of Mathematical Sciences, Faculty of Science, Yamagata University, Kojirakawa-machi 1-4-12, Yamagata 990-8560, Japan.}
\email{s.fukasawa@sci.kj.yamagata-u.ac.jp} 
\thanks{The second author was partially supported by JSPS KAKENHI Grant Number 16K05088.}  
\newtheorem{theorem}{Theorem}
\newtheorem{proposition}{Proposition} 
\newtheorem{corollary}{Corollary}

\newtheorem{fact}{Fact}
\newtheorem{problem}{Problem} 

\theoremstyle{definition}
 
\newtheorem{remark}{Remark}

\begin{document}
\begin{abstract} 
We determine the distribution of Galois points for plane curves over a finite field of $q$ elements, which are Frobenius nonclassical for different powers of $q$. 
This family is an important class of plane curves  with many  remarkable properties. It contains the Dickson--Guralnick--Zieve curve, which has been recently studied by Giulietti, Korchm\'{a}ros, and Timpanella from  several points of view. 
A problem posed by the second author in the theory of Galois points is modified. 
\end{abstract}
\maketitle

\section{Introduction}  
Let $\mathbb F_q$ be a finite field with $q \ge 2$, and let $\mathcal{F} \subset \mathbb{P}^2$ be the plane curve defined by $F(x, y, z)=D_1(x, y, z)/D_2(x, y, z)$, where 
\begin{equation} \label{defining equation} 
D_1=\left|\begin{array}{ccc}
x & x^{q^m} & x^{q^n} \\
y & y^{q^m} & y^{q^n} \\ 
z & z^{q^m} & z^{q^n} 
\end{array} \right| \  \mbox{ and } \ 
D_2=\left|\begin{array}{ccc}
x & x^q & x^{q^2} \\
y & y^q & y^{q^2} \\ 
z & z^q & z^{q^2} 
\end{array} \right|,
\end{equation} 
and $n$ and $m$ are coprime.
According to \cite[p.544 and Theorem 3.4]{borges},  $F$ is a homogeneous polynomial of degree $q^n+q^m-q^2-q$ over $\mathbb{F}_q$, which   is irreducible over the algebraic closure $\overline{\mathbb F}_q$. 
In 2009, the first author characterized these curves as the unique double-Frobenius nonclassical plane curves for different powers $q^n$ and $q^m$, with ${\rm gcd}(n, m)=1$. Other significant features, such as a large number of $\mathbb{F}_{q^n}$-rational points (meeting the St\"ohr--Voloch bound) and the arc property  in the case $m=1$
 were also noted   \cite{borges}. 
This  important family of curves  contains the  newly coined  Dickson--Guralnick--Zieve (DGZ) curve: case  $(n, m)=(3,1)$. 
Additional remarkable properties of the DGZ  curve, such as a large automorphism group and positive $p$-rank,  have  been recently proved  by Giulietti, Korchm\'{a}ros, and Timpanella \cite{gkt}. 

In this article, we consider Galois points for the curves $\mathcal{F}$ over $\overline{\mathbb{F}}_q$ (see \cite{miura-yoshihara, yoshihara} for the definition of Galois point). 
The set of all Galois points for the plane curve $\mathcal{F}$ on the projective plane is denoted by $\Delta(\mathcal{F})$. 

Our main result is the following. 

\begin{theorem} 
Let $n \ge 3$ and $m \ge 1$ be integers such that $n>m$ and ${\rm gcd}(n, m)=1$. 
For the $(q^n, q^m)$-Frobenius nonclassical curve $\mathcal{F} \subset \mathbb{P}^2$, 
$$\Delta(\mathcal{F})=\emptyset \ \mbox{ or  } \ \mathbb{P}^2(\mathbb{F}_q).$$ 
The latter case occurs if and only if $(n, m)=(3,1)$ or $(3,2)$. 
\end{theorem}

Since the result for the case where $(n, m)=(3,1)$ or $(3, 2)$ gives a negative answer to the problem \cite[Problem 1]{fukasawa2}, posed by the second author,  it  is modified as follows. 

\begin{problem} 
Let $\mathcal{C}$ be a plane curve over $\mathbb F_q$. 
Assume that $\Delta(\mathcal{C})=\mathbb P^2(\mathbb F_q)$. 
Then, is it true that $\mathcal{C}$ is projectively equivalent to the Hermitian, Ballico--Hefez or the $(q^n, q^m)$-Frobenius nonclassical curve of type $(n, m)=(3, 1)$ or $(3, 2)$?
Or, more basically, is $\mathcal{C}$ Frobenius nonclassical?
\end{problem} 

The result in \cite{fukasawa3} is contained in this article.   

\section{Preliminaries}
Let $\mathcal{C} \subset \mathbb{P}^2$ be an irreducible plane curve of degree $d$, and let $r:\hat{\mathcal{C}} \rightarrow \mathcal{C}$ be the normalization. 
For different points $P$ and $Q \in \mathbb{P}^2$, the line passing through $P$ and $Q$ is denoted by $\overline{PQ}$. 
For a point $P \in \mathbb{P}^2$, $\pi_P:\mathcal{C} \dashrightarrow \mathbb{P}^1$ represents the projection from $P$.  
The composite map $\pi_P \circ r: \hat{\mathcal{C}} \rightarrow \mathbb{P}^1$ is denoted by $\hat{\pi}_P$. 
The ramification index at $\hat{Q} \in \hat{\mathcal{C}}$ is represented by $e_{\hat{Q}}$. 
When $r^{-1}(Q)$ consists of a unique point $\hat{Q} \in \hat{\mathcal{C}}$, the index $e_{\hat{Q}}$ is denoted also by $e_Q$.

\begin{fact} \label{ramification}
For the projection $\hat{\pi}_P$, the following holds. 
\begin{itemize}
\item[(a)] For each point $\hat{Q} \in \hat{C}$ with $Q=r(\hat{Q}) \ne P$, it follows that $e_{\hat{Q}}={\rm ord}_{\hat{Q}}h_{PQ}$, where $h_{PQ}$ is a linear polynomial defining the line $\overline{PQ}$. 
\item[(b)] If $P \in \mathcal{C} \setminus {\rm Sing}(\mathcal{C})$, then $e_{P}={\rm ord}_{P}h_{P}-1$, where $h_{P}$ is a linear polynomial defining the tangent line at $P$. 
\item[(c)] If $P$ is an ordinary singularity of $\mathcal{C}$ with multiplicity $m(P)$, that is, $P$ has $m(P)$ tangent lines, then for each tangent line defined by $h=0$ at $P$, there exists a unique point $\hat{P} \in r^{-1}(P)$ such that $e_{\hat{P}}={\rm ord}_{\hat{P}}h-1$.  
\end{itemize}
\end{fact}

According to \cite[Proposition 3.2]{borges}, the following holds. 

\begin{fact} \label{singularities}
Let $\mathcal{F} \subset \mathbb{P}^2$ be the $(q^n, q^m)$-Frobenius nonclassical curve, and let $S \subset \mathbb{P}^2$ be the set of all points $R$ such that $R$ is contained in some $\mathbb{F}_q$-line. 
\begin{itemize}
\item[(a)] For the case where $m>1$, ${\rm Sing}(\mathcal{F})=\mathbb{P}^2(\mathbb{F}_{q^{n-m}})$.  
Let $Q=r(\hat{Q}) \in {\rm Sing}(\mathcal{F})$ be of multiplicity $m(Q)$. 
\begin{itemize}
\item[(i)] If $Q \not\in S$, then $m(Q)=q^m$, $r^{-1}(Q)=\{\hat{Q}\}$, and there exists a unique line $T_Q \ni Q$ such that ${\rm ord}_{\hat{Q}}T_Q=q^m+1$. 
\item[(ii)] If $Q \in S$ and $Q \not\in \mathbb{P}^2(\mathbb{F}_q)$, then $m(Q)=q^m-1$, $r^{-1}(Q)=\{\hat{Q}\}$, and there exists a unique line $T_Q \ni Q$ such that ${\rm ord}_{\hat{Q}}T_Q=q^m$. 
\item[(iii)] If $Q \in \mathbb{P}^2(\mathbb{F}_q)$, then $m(Q)=q^m-q$, $r^{-1}(Q)$ consists of exactly $m(Q)$ points, and $Q$ is an ordinary singularity. 
For each tangent line $T$ at $Q$, the intersection multiplicity of $\mathcal{F}$ and $T$ at $Q$ is equal to $q^n-q$. 
\end{itemize} 
\item[(b)] For the case where $m=1$ and $q>2$, ${\rm Sing}(\mathcal{F})=\mathbb{P}^2(\mathbb{F}_{q^{n-1}}) \setminus \mathbb{P}^2(\mathbb{F}_q)$. 
Let $Q=r(\hat{Q}) \in {\rm Sing}(\mathcal{F})$ be of multiplicity $m(Q)$. 
\begin{itemize}
\item[(i)] If $Q \not\in S$, then $m(Q)=q$, $r^{-1}(Q)=\{\hat{Q}\}$, and there exists a unique line $T_Q \ni Q$ such that ${\rm ord}_{\hat{Q}}T_Q=q+1$. 
\item[(ii)] If $Q \in S$ and $Q \not\in \mathbb{P}^2(\mathbb{F}_q)$, then $m(Q)=q-1$, $r^{-1}(Q)=\{\hat{Q}\}$, and there exists a unique line $T_Q \ni Q$ such that ${\rm ord}_{\hat{Q}}T_Q=q$. 
\end{itemize}
\item[(c)] For the case where $(m, q)=(1, 2)$, ${\rm Sing}(\mathcal{F})=\mathbb{P}^2(\mathbb{F}_{2^{n-1}})\setminus S$.  
If $Q=r(\hat{Q}) \in {\rm Sing}(\mathcal{F})$ with multiplicity $m(Q)$, then $m(Q)=2$, $r^{-1}(Q)=\{\hat{Q}\}$, and there exists a unique line $T_Q \ni Q$ such that ${\rm ord}_{\hat{Q}}T_Q=3$. 
\end{itemize}
\end{fact}

If $P$ is a Galois point, then the following holds (see \cite[III.7.2]{stichtenoth}). 

\begin{fact} \label{galois-ramification}
If the projection $\hat{\pi}_P: \hat{\mathcal{C}} \rightarrow \mathbb{P}^1$ is a Galois covering, then we have the following.
\begin{itemize}
\item[(a)] If $\hat{Q}_1$ and $\hat{Q}_2 \in \hat{\mathcal{C}}$ have the same image, then $e_{\hat{Q}_1}=e_{\hat{Q}_2}$.  
\item[(b)] For each point $\hat{Q} \in \hat{\mathcal{C}}$, the index $e_{\hat{Q}}$ divides the degree $\deg \hat{\pi}_{P}$. 
\end{itemize}
\end{fact} 

Combining Facts \ref{ramification}(a), \ref{singularities} and \ref{galois-ramification}(a), we have the following. 

\begin{corollary} \label{singular ramification}
Let $\hat{\pi}_P: \hat{\mathcal{F}} \rightarrow \mathbb{P}^1$ be the projection, and let $Q=r(\hat{Q}) \in \mathcal{F} \setminus \{P\}$ be a singular point. 
\begin{itemize}
\item[(i)] If $Q$ is in the case (a)(i), (b)(i), or (c) of Fact \ref{singularities}, then $e_{\hat{Q}}=q^m$ or $q^m+1$. 
\item[(ii)] If $Q$ is in the case (a)(ii) or (b)(ii) of Fact \ref{singularities}, then $e_{\hat{Q}}=q^m-1$ or $q^m$. 
\item[(iii)] Assume that $P$ is a Galois point. 
If $Q$ is in the case (a)(iii) of Fact \ref{singularities}, then $e_{\hat{Q}}=1$. 
\end{itemize}
\end{corollary}

\section{Frobenius nonclassicality and Galois points}

\begin{proposition} \label{key prop}
Let $\mathcal{C} \subset \mathbb{P}^2$ be a $q$-Frobenius nonclassical curve over $\mathbb{F}_q$. 
Assume that $P$ is a Galois point for $\mathcal{C}$ and $Q \in \mathcal{C} \setminus {\rm Sing}(\mathcal{C})$ is a ramification point for the projection from $P$. 
\begin{itemize}
\item[(a)] If $P \in (\mathbb{P}^2 \setminus \mathcal{C}) \cup {\rm Sing}(\mathcal{C})$, then the line $\overline{PQ}$ is defined over $\mathbb{F}_q$.  
\item[(b)] If $P=Q$, then the tangent line $T_P\mathcal{C}$ at $P$ is defined over $\mathbb{F}_q$. 
\item[(c)] If $P \in \mathcal{C} \setminus {\rm Sing}(\mathcal{C})$ and there exists a point $Q' \in (\mathcal{C} \setminus {\rm Sing}(\mathcal{C})) \cap (\overline{PQ} \setminus \{P, Q\})$, then the line $\overline{PQ}$ is defined over $\mathbb{F}_q$.  
\end{itemize}
\end{proposition}

\begin{proof} 
We prove assertions (a) and (c). 
Let $Q^q$ and $L^q$ be the $q$-Frobenius images of the point $Q$ and the line $L:=\overline{PQ}$, respectively. 
Considering Fact \ref{galois-ramification}(a), we need only prove the claim under the assumption that $Q^q \ne P$.  
Since $Q$ is a ramification point, it follows from Fact \ref{ramification}(a) that the line $L$ is tangent to $\mathcal{C}$ at $Q$. 
It follows that $L^q$ is tangent to $\mathcal{C}$ at $Q^q$.  
Now the $q$-Frobenius nonclassicality of $\mathcal{C}$ implies that $Q^q$ lies on $L$. 
The assumption $Q^q \ne P$ and Fact \ref{galois-ramification}(a) imply that $L$ is tangent to $\mathcal{C}$ at $Q^q$.
Hence $L=L^q$, that  is, $L$ is an $\mathbb{F}_q$-line.

We consider the case where $P$ is a ramification point of $\hat{\pi}_P$. 
By the Frobenius nonclassicality, $P^q \in T_P\mathcal{C}$. 
If $P^q \ne P$, then Fact \ref{galois-ramification}(a) implies that $P^q$ is a ramification point.  
Therefore, the tangent line $T_{P^q}\mathcal{C}$ is the same as $T_P\mathcal{C}$. 
Similar to the above proof, $T_P\mathcal{C}$ is $\mathbb{F}_q$-rational. 
\end{proof}

\begin{corollary} \label{key cor}
Let $\mathcal{F} \subset \mathbb{P}^2$ be the $(q^n, q^m)$-Frobenius nonclassical curve over $\mathbb{F}_q$. 
Assume that $P$ is a Galois point for $\mathcal{F}$ and $Q \in \mathcal{F} \setminus {\rm Sing}(\mathcal{F})$ is a ramification point for the projection from $P$. 
\begin{itemize}
\item[(a)] If $P \in (\mathbb{P}^2 \setminus \mathcal{F}) \cup {\rm Sing}(\mathcal{F})$, then the line $\overline{PQ}$ is defined over $\mathbb{F}_q$.  
\item[(b)] If $P=Q$, then the tangent line $T_P\mathcal{F}$ at $P$ is defined over $\mathbb{F}_q$.
\item[(c)] If $P \in \mathcal{F} \setminus {\rm Sing}(\mathcal{F})$ and there exists a point $Q' \in (\mathcal{F} \setminus {\rm Sing}(\mathcal{F})) \cap (\overline{PQ} \setminus \{P, Q\})$, then the line $\overline{PQ}$ is defined over $\mathbb{F}_q$.
\end{itemize} 
\end{corollary}

\begin{proof}
By Proposition \ref{key prop}, the line $\overline{PQ}$ is $\mathbb{F}_{q^n}$-rational and $\mathbb{F}_{q^m}$-rational.
Since $n$ and $m$ are coprime, the line is defined over $\mathbb{F}_{q}$. 
\end{proof}

\section{Inner smooth Galois points} 
Let $P \in \mathcal{F} \setminus {\rm Sing}(\mathcal{F})$ be an inner Galois point. 
It follows from Fact \ref{ramification}(b) that $e_P=I_P(\mathcal{F}, T_P\mathcal{F})-1$ for the projection $\hat{\pi}_P$, where $I_P(\mathcal{F}, T_P\mathcal{F})$ is the intersection multiplicity of the curve $\mathcal{F}$ and the tangent line $T_P\mathcal{F}$ at $P$. 
Assume that $(m, q) \ne (1, 2)$. 
Then $I_P(\mathcal{F}, T_P\mathcal{F}) \ge q^m \ge 3$ (\cite[Theorem 2.6]{borges}), and hence $P$ is a ramification point.   
By Corollary \ref{key cor}(b), the tangent line $T_P\mathcal{F}$ is $\mathbb{F}_{q}$-rational.  
If $m>1$, then there exists an ordinary singularity on $T_P\mathcal{F}$ (Fact \ref{singularities}(a)), by Corollary \ref{singular ramification}(iii), and this is a contradiction to Fact \ref{galois-ramification}(a). 
If $m=1$ and $q>2$, then $T_P\mathcal{F}$ contains a singular point $Q$ with index $e_Q=q$ or $q-1$ (Corollary \ref{singular ramification}(ii)). 
It follows from Fact \ref{galois-ramification}(b) that $e_Q$ divides $\deg \hat{\pi}_P=q^n-q^2-1$. 
This is impossible. 

Assume that $(m, q)=(1, 2)$ and $n>3$. 
Since $n-m>2$, there exists a singular point $Q$ such that there does not exist an $\mathbb{F}_2$-line containing $Q$ (Fact \ref{singularities}(c)). 
Therefore, $\overline{PQ}$ is not an $\mathbb{F}_2$-line. 
Since $\deg \hat{\pi}_P=q^n-q^2-1$, it follows from Corollary \ref{singular ramification}(i) and Fact \ref{galois-ramification}(b) that $e_Q=3$. 
Since the tangent line at the point $Q$ is $\mathbb{F}_{2^{n-1}}$-rational, the line $\overline{PQ}$ is $\mathbb{F}_{2^{n-1}}$-rational. 
There exist at least two $\mathbb{F}_2$-lines intersecting $\overline{PQ}$ at points of $\mathcal{F}$ different from $P$. 
It follows from Fact \ref{singularities}(c) that such points are smooth points. 
According to Corollary \ref{key cor}(c), the line $\overline{PQ}$ is $\mathbb{F}_q$-rational. 
This is a contradiction. 

\section{The case where $n-m>2$} 

Assume that $n-m>2$ and $P \in (\mathbb{P}^2 \setminus \mathcal{F}) \cup {\rm Sing}(\mathcal{F})$ is a Galois point. 
Since $n-m>2$, there exists a singular point $Q \ne P$ not contained in any $\mathbb{F}_q$-line (Fact \ref{singularities}). 
Then $Q$ is a ramification point for $\hat{\pi}_P$ with index $e_Q=q^m$ or $q^m+1$ and the line $\overline{PQ}$ is not $\mathbb{F}_q$-rational. 
By Corollary \ref{key cor}(a), $\overline{PQ}$ does not contain a smooth point. 
It follows from Fact \ref{galois-ramification}(a) and Corollary \ref{singular ramification} that $\overline{PQ}$ is an $\mathbb{F}_{q^{n-m}}$-line and $\overline{PQ} \cap \mathcal{F}$ consists of only singular points. 
By considering the intersection points given by $\overline{PQ}$ and $\mathbb{F}_q$-lines, there exists a point $Q' \in \overline{PQ}$ such that $Q'$ is not $\mathbb{F}_q$-rational and is contained in some $\mathbb{F}_q$-line. 
It follows from Proposition \ref{singular ramification} that $e_{Q'}=q^m-1$ or $q^m$. 
By Fact \ref{galois-ramification}(a), $e_{Q'}=e_Q$ and hence  $e_{Q'}=e_Q=q^m$. 
Note that the number of points in $\overline{PQ} \cap \mathcal{F}$ is equal to $q^{n-m}+1$ or $q^{n-m}$. 

Assume that $P \in \mathbb{P}^2 \setminus \mathcal{F}$. 
Then $q^n+q^m-q^2-q=e_Q(q^{n-m}+1)$, or $e_Qq^{n-m}$. 
This is impossible. 

Assume that $P \in {\rm Sing}(\mathcal{F}) \setminus \mathbb{P}^2(\mathbb{F}_q)$.  
Since $e_Q=q^m$ divides $\deg \hat{\pi}_P$, by Fact \ref{singularities}, the multiplicity $m(P)$ is equal to $q^m$. 
Then  $q^n-q^2-q=e_Q(q^{n-m}+1)$, $e_Qq^{n-m}$, or $e_Q(q^{n-m}-1)$. 
This is also impossible. 

Assume that $P \in {\rm Sing}(\mathcal{F}) \cap \mathbb{P}^2(\mathbb{F}_q)$. 
By Fact \ref{singularities}, $m>1$. 
It follows from Fact \ref{singularities}(a) that $\deg \hat{\pi}_P=q^n-q^2$ and $(\overline{PQ} \setminus \{P\}) \cap \mathcal{F}$ consists of exactly $q^{n-m}$ singular points.  
According to \cite[Remark 3.3]{borges}, $\overline{PQ}$ is a tangent line at $P$. 
Since $P$ is an ordinary singularity, by Fact \ref{ramification}(c), the fiber of the point corresponding to $\overline{PQ}$ for $\hat{\pi}_P$ contains exactly $q^{n-m}+1$ points.  
It follows that $e_Q(q^{n-m}+1)=q^n-q^2$. 
This is impossible.

\section{The case where $n-m=2$}

Assume that $n-m=2$, $(m, q) \ne (1, 2)$ and $P \in (\mathbb{P}^2 \setminus \mathcal{F}) \cup {\rm Sing}(\mathcal{F})$ is a Galois point. 
Let $Q$ be a singular point different from $P$ that is contained in $\mathbb{P}^2(\mathbb{F}_{q^2}) \setminus \mathbb{P}^2(\mathbb{F}_q)$ (Fact \ref{singularities}). 
Since any $\mathbb{F}_{q^2}$-line contains an $\mathbb{F}_q$-point, $Q$ is a ramification point with index $q^m-1$ or $q^m$. 

Assume that $m>1$. 
It follows from Fact \ref{galois-ramification}(a) and Corollary \ref{singular ramification}(iii) that $\overline{PQ}\setminus \{P\}$ does not contain an ordinary singularity. 
Therefore, $\overline{PQ}\setminus \{P\}$ does not contain an $\mathbb{F}_q$-rational point. 
In particular, $\overline{PQ}$ is not $\mathbb{F}_q$-rational. 
By Corollary \ref{key cor}(a), $\overline{PQ}$ does not contain a smooth point. 
It follows from Fact \ref{galois-ramification}(a) and Corollary \ref{singular ramification} that $\overline{PQ}$ is an $\mathbb{F}_{q^2}$-line. 
Since any $\mathbb{F}_{q^2}$-line contains an $\mathbb{F}_q$-point, the point $P$ must be an $\mathbb{F}_q$-point, and hence a point with multiplicity $q^m-q$. 
Then $\overline{PQ}$ is a tangent line at $P$. 
Furthermore, the set $(\mathcal{F} \cap \overline{PQ}) \setminus \{P\}$ consists of exactly $q^{2}$ $\mathbb{F}_{q^{2}}$-rational but not $\mathbb{F}_q$-rational points. 
Note that there exists a unique point $\hat{P} \in r^{-1}(P)$ such that the image $\hat{\pi}_P(\hat{P})$ corresponds to the line $\overline{PQ}$. 
Therefore, the fiber $\hat{\pi}_P^{-1}(\overline{PQ})$ contains exactly $q^{2}+1$ points. 
It follows that $e_Q(q^{2}+1)=q^n-q^2$. 
This is impossible. 

Assume that $m=1$. 
Then $n=3$. 
If $\overline{PQ}$ contains a smooth point, then by Corollary \ref{key cor}(a), $\overline{PQ}$ is $\mathbb{F}_q$-rational. 
If $\overline{PQ}$ does not contain a smooth point, then it follows from Fact \ref{galois-ramification}(a) and Corollary \ref{singular ramification} that $\overline{PQ}$ is $\mathbb{F}_{q^2}$-rational. 
Therefore, $P$ is $\mathbb{F}_{q^2}$-rational. 
Assume that $P \in \mathbb{P}^2(\mathbb{F}_{q^2}) \setminus \mathbb{P}^2(\mathbb{F}_q)={\rm Sing}(\mathcal{F})$. 
Note that the tangent line $T$ at $P$ is defined over $\mathbb{F}_q$, and the tangent line at each singular point $R$ in $T \cap (\mathbb{P}^2(\mathbb{F}_{q^2}) \setminus \mathbb{P}^2(\mathbb{F}_q))$ is the same as $T$.  
Then  $e_R=q$. 
It follows from Fact \ref{galois-ramification}(b) that $e_R$ divides $\deg \hat{\pi}_P=q^3-q^2-(q-1)$. 
This is a contradiction. 

In conclusion, it follows that if $\Delta(\mathcal{F}) \ne \emptyset$, then $m=1$ and $\Delta(\mathcal{F}) \subset \mathbb{P}^2(\mathbb{F}_q) \subset \mathbb{P}^2 \setminus \mathcal{F}$. 

\begin{remark}
When $(n, m, q)=(3, 1, 2)$, the curve $\mathcal{F}$ is given by 
$$ F(x, y, z)=(x^2+xz)^2+(x^2+xz)(y^2+yz)+(y^2+yz)^2+z^4$$
(see \cite[p.542]{borges} or \cite[Remark 1]{gkt}). 
In this case, it is known that the claim follows (\cite[Theorem 4]{fukasawa1}). 
\end{remark} 

\section{The case where $n-m=1$}

Assume that $n-m=1$ and $P \in \mathbb{P}^2 \setminus \mathcal{F}$ is a Galois point. 
Since all singular points are ordinary singularities (Fact \ref{singularities}), by Corollary \ref{singular ramification}(iii), the projection from $P$ is not ramified at such points. 
Therefore, there exists a smooth point $Q \in \mathcal{F}$ that  is a ramification point. 
By Corollary \ref{key cor}(a), the line $\overline{PQ}$ is $\mathbb{F}_q$-rational. 
However, there exists an ordinary singularity in $\overline{PQ}$. 
This is a contradiction to Fact \ref{galois-ramification}(a). 

Assume that $n-m=1$, $m>2$ and $P \in {\rm Sing}(\mathcal{F})$ is a Galois point. 
Let $T$ be a tangent line at $P$. 
By Facts \ref{ramification}(c) and \ref{singularities}(a)(iii), there exists a ramification point $\hat{P} \in r^{-1}(P)$ contained in the fiber of the point corresponding to $T$ for the projection $\hat{\pi}_P$. 
Since $m>2$, there exists a point $Q \in (\mathcal{F} \cap T) \setminus \{P\}$ (Fact \ref{singularities}(a)(iii)). 
By \cite[Remark 3.3]{borges}, since $T$ is not an $\mathbb{F}_q$-line, $Q$ is a smooth point. 
By Fact \ref{galois-ramification}(a), $Q$ is a ramification point. 
It follows from Corollary \ref{key cor}(a) that $T$ is $\mathbb{F}_q$-rational.
This is a contradiction.

Accordingly, it follows that if $\Delta(\mathcal{F}) \ne \emptyset$, then $m=2$ and $\Delta(\mathcal{F}) \subset {\rm Sing}(\mathcal{F})=\mathbb{P}^2(\mathbb{F}_q)$.

\section{The case where $(n, m)=(3,1)$ or $(3,2)$}

It is easily verified that the projective linear group $PGL(3, \mathbb{F}_q)$ acts on $\mathcal{F}$ by the definitions of $D_1$ and $D_2$ as in (\ref{defining equation}) (see also \cite[Lemma 4.1]{gkt}). 
Therefore, the matrices
$$\sigma_{\gamma, \beta}=\left(\begin{array}{ccc}
1 & 0 & 0 \\
\gamma & 1 & \beta \\ 
0 & 0 & 1 
\end{array} \right) \ \mbox{ and } \ 
\tau_{\mu}=\left(\begin{array}{ccc}
1 & 0 & 0 \\
0 & \mu & 0 \\ 
0 & 0 & 1 
\end{array} \right) \in PGL(3, \mathbb{F}_q)$$
act on $\mathcal{F}$, where $\gamma, \beta \in \mathbb{F}_q$ and $\mu^{q-1}=1$. 
Note that a rational function $x/z$ is fixed by the actions of $\sigma_{\gamma, \beta}$ and $\tau_{\mu}$. 
This implies that $\pi_{P} \circ \sigma_{\gamma, \beta}=\pi_{P}$ and $\pi_{P} \circ \tau_{\mu}=\pi_{P}$, where $P=(0:1:0)$. 
Note also that $\deg \pi_P=q^3-q^2$ if and only if $(n, m)=(3, 1)$ or $(3, 2)$ (Fact \ref{singularities}). 
Considering the action of $PGL(3, \mathbb{F}_q)$, it is inferred that $\mathbb{P}^2(\mathbb{F}_q) \subset \Delta(\mathcal{F})$ holds, if $(n, m)=(3, 1)$ or $(3, 2)$.

\begin{remark}
\begin{itemize}
\item[(1)] If $(n, m)=(3, 1)$ or $(3, 2)$, then the associated Galois group $G_P$ of a Galois point $P \in \Delta(\mathcal{F})$ is isomorphic to the semidirect product $\mathbb{F}_q^{\oplus 2} \rtimes \mathbb{F}_q^*$, where the action of $\mathbb{F}_q^*$ on $\mathbb{F}_q^{\oplus 2}$ is given by $(\mu, (\gamma,  \beta)) \mapsto (\mu \gamma, \mu \beta)$. 
\item[(2)] All points in $\mathbb{P}^2(\mathbb{F}_q)$ are quasi-Galois points which are not Galois points, if $(n, m) \ne (3,1)$ or $(3,2)$ (see \cite{fmt} for the definition of quasi-Galois point). 
\end{itemize}
\end{remark}

\end{document}